\documentclass[11pt]{amsart}
\usepackage{mathrsfs}
\usepackage{amsfonts}
\usepackage{amsfonts,latexsym,rawfonts,amsmath,amssymb,amsthm}
\usepackage[plainpages=false]{hyperref}

\usepackage{graphicx}

\numberwithin{equation}{section}

\newcommand{\beq}{\begin{equation}}
\newcommand{\eeq}{\end{equation}}
\newcommand{\beqs}{\begin{eqnarray*}}
\newcommand{\eeqs}{\end{eqnarray*}}
\newcommand{\beqn}{\begin{eqnarray}}
\newcommand{\eeqn}{\end{eqnarray}}
\newcommand{\beqa}{\begin{array}}
\newcommand{\eeqa}{\end{array}}

\def\lra{\longrightarrow}

\def\p{\prime}

\def\bc{\begin{center}}
\def\ec{\end{center}}

\def\p{\partial}

\def\l{\lambda}

\def\S{\Sigma}

\def\ZZ{{\mathbb Z}}
\def\QQ{{\mathbb Q}}
\def\RR{{\mathbb R}}

\def\CC{{\mathbb C}}

\def\begeq{\begin{equation}}
\def\endeq{\end{equation}}
\def\and{\quad{\rm and}\quad}

\let\lra=\longrightarrow

\def\mapright\#1{\,\smash{\mathop{\lra}\limits^{\#1}}\,}

\def\pbp{\sqrt{-1}\partial\bar\partial}

\newtheorem{prop}{Proposition}[section]
\newtheorem{theo}[prop]{Theorem}
\newtheorem{lem}[prop]{Lemma}

\newtheorem{rem}[prop]{Remark}
\newtheorem{ex}[prop]{Example}
\newtheorem{defi}[prop]{Definition}

\begin{document}

\title{K\"ahler-Ricci solitons on toric Fano orbifolds}


\author{ Yalong Shi}
\author{ Xiaohua Zhu*}

\address{Yalong Shi\\Department of Mathematics, Nanjing University,
 Nanjing, 210093, Jiangsu Province, China\\
shiyl@nju.edu.cn}

\address{ Xiaohua Zhu \\School of Mathematical Sciences, Peking University,
Beijing, 100871, China\\
 xhzhu@math.pku.edu.cn}

\thanks {  *Partially supported by  NSF10990013 in China.}

\maketitle

\begin{abstract}
  We prove the existence of K\"ahler-Ricci solitons on toric Fano orbifolds,
  hence extend the theorem of Wang and Zhu \cite{WZ} to the orbifold case.
\end{abstract}


\section{Introduction}\label{section:intro}

A complex orbifold of dimension $n$ is a Haudorff space $X$ with a
family of local uniformizing charts $\{(\tilde U, G, \varphi)\}$.
Here, $\tilde U$ is an open subset of $\CC^n$, $G$ is a finite group
of bi-holomorphic transformations of $\tilde U$, and $\varphi$ is a
continuous map from $\tilde U$ to an open set $U\subset X$, such
that it induces a homeomorphism $\tilde U/ G \to U$. The notion of
orbifold was first introduced by Satake (\cite{Sa1} ) in the name of
``V-manifold" in the 1950's. An orbifold is a generalization of
manifold, and we can define orbifold-smooth functions and maps by
requiring that the corresponding liftings to the local uniformizing
charts can be extended smoothly to the whole of the charts. Many
theorems on manifolds, like Hodge decomposition theorem and Kodaira
imbedding theorem, were generalized to orbifolds
by Baily and Satake in \cite{Bai1}, \cite{Bai2} and \cite{Sa2}. The analysis
on orbifolds is also studied by many people, for example, see \cite{Ch}
for a discussion of Sobolev spaces.\\

On a complex orbifold $X$, one can define orbifold K\"ahler metrics
and corresponding Ricci forms as on manifolds. The Ricci form is a
closed form and hence defines a cohomology class in the Dolbeault
group, and we call it the first Chern class of $X$, denoted by
$c_1(X)$.

\begin{defi}
 A complex normal variety $X$ with only orbifold singularities is called
Fano if the Weil divisor $-K_X$ is an ample $\QQ$-Cartier divisor,
i.e. a multiple of $-K_X$ is ample Cartier. Equivalently, ( by
Baily's embedding theorem \cite{Bai2}) $X$ is called a Fano orbifold
if one can represent $c_1(X)$ by an orbifold K\"ahler form. $X$ is
called Gorenstein Fano if  $-K_X$ itself is an ample Cartier
divisor.
\end{defi}

 Like on manifolds,
a fundamental problem in differential gemoetry of orbifolds is the
existence of canonical metrics, like
 the Einstein metrics. In this paper, we study the existence of  K
 \"ahler-Ricci solitons on  Fano orbifolds.
\begin{defi}
Let $X$ be a Fano orbifold and $\omega_g$ be the K\"ahler form of a
K\"ahler metric $g$ on $X$ with $\frac{1}{2\pi}[\omega_g]=c_1(X)>0$,
where
$$\omega_g:=\sqrt{-1}g_{i\bar j}dz^i\wedge d\bar z^j.$$
$\omega_g$  is called a K\"ahler-Ricci soliton if there is an
orbifold holomorphic vector field $v$ such that
\beq\label{eqn:soliton} Ric(\omega_g)-\omega_g=L_v \omega_g, \eeq
where $Ric(\omega_g)$ is the Ricci form of $\omega_g$ defined by
$$Ric(\omega_g):=-\pbp \log \det (g_{i\bar j}).$$
\end{defi}

Here by ``an orbifold holomorphic vector field", we mean a
holomorphic vector field $v$ on the regular part $X_{reg}$, such
that for any local uniformizing chart $\pi:\tilde U\to U\subset X$,
the lifting of $v$ extends to the whole of $\tilde U$ as a
holomorphic vector field on $\tilde U$, see \cite{Bai2}.  The Lie
derivative of a form $\eta$ with repect to a complex vector field
$v$ is defined by the Cartan formula
$$L_v \eta:= di_v \eta +i_v d\eta.$$

\begin{rem}
  Since the singular set of a normal variety always has codimension at least 2, by a standard extension theorem
  in complex analysis, any holomorphic vector field $v$ on the
  $X_{reg}$ of a normal orbifold $X$  is an orbifold holomorphic vector field on $X$.
\end{rem}

The K\"ahler-Ricci soliton is a generalization of K\"ahler-Einstein
metric, it is also conjectured to be the limit of K\"ahler-Ricci
flow (\cite{T2}). The main result of this paper is the following
theorem:

\begin{theo}\label{thm:main}
 For any toric Fano orbifold $X$, there exists a $T$-invariant K\"ahler-Ricci soliton metric, the soliton metric is  Einstein
if and only if the Futaki invariant of $X$ vanishes.
\end{theo}

\begin{rem}
  The uniqueness theorem of Tian and Zhu in \cite{TZ1} and
  \cite{TZ2} should also hold for Fano orbifolds. The uniqueness of K\"ahler-Einstein metrics
  on Fano orbifolds is true indeed.
\end{rem}

For toric Fano manifolds, the existence of K\"ahler-Ricci solitons
was proved by Wang and Zhu in \cite{WZ}. In \cite{Na}, Nakagawa also
solved the existence problem of K\"ahler-Einstein metrics on toric
Gorenstein Fano orbifolds in dimension 2. He also conjectures that
the vanishing of Futaki invariant should be a sufficient condition
for the existence of K\"ahler-Einstein metrics on toric Gorenstein
Fano orbifolds and in this case the automorphism group is reductive.
Our Theorem \ref{thm:main} confirms a generalized version of his conjecture.\\

The organization of this paper is as follows. In section
\ref{section:Guillemin}, we review relevant results on toric
varieties. In section \ref{section:equation}, we use the torus
action to rewrite the soliton equation as a real Monge-Amp\`ere
equation. Finally in section \ref{section:existenc}, we prove the
main theorem by establishing various a priori estimates along the
same lines as \cite{WZ}. Then we give two examples of toric Fano orbifolds.\\

{\bf Acknowledgements } This work was partially done  when both
authors were visiting the math department of USTC in the fall of
2010. The authors would like to thank Professor Chen Xiuxiong for
his invitation. The first named author also thanks Professor Xu Bin,
Zhang Wei
and Doctor Xu Jinxing for helpful discussions.\\

\section{K\"ahler metrics on a toric Fano
variety}\label{section:Guillemin}

Let $N\cong \ZZ^n$ be a lattice, and $Q\subset N_\RR\cong \RR^n$ a
convex lattice polytope, i.e., the vertices of $Q$ are element of
$N$. Suppose $Q$ contains the origin in its interior, then the cones
over the faces of $Q$ form a complete fan. The toric variety $X_Q$
associated to this fan is a normal projective variety. When all the
vertices of $Q$ are primitive and all the faces of $Q$ are
simplicial, then $X_Q$ is a Fano orbifold. Conversely, all the toric
Fano orbifolds are obtained in this way (see \cite{Deb}). In
particular, for any $n\geq 2$, there are infinitely many isomorphism
classes of
toric Fano orbifolds of dimension $n$.\\

Let $X_Q$ be a toric Fano orbifold as above. Let $M:=Hom_\ZZ(N,\ZZ)$
be the dual of $N$, and $P\subset M_{\RR}$ the dual of $Q$ defined
by
$$P:=\{y\in M_{\RR}|~~\langle y,x\rangle\geq -1, \forall x\in Q\}.$$
The polytope also contains the origin in its interior, and $X_Q$ is
Gorenstein Fano if and only if $P$ is also a lattice polytope. The
isomorphism classes of toric Gorenstein Fano varieties of any
dimension are finite. For dimension 2, there are exact 16 different
classes, see \cite{Na} or Chapter 8 of \cite{CLS}.\\

From a differential geometric viewpoint, the dual polytope $P$ is
more important. Since the faces of $Q$ are simplicial, at every
vertex of $P$ there are precisely $n$ facets meeting at this vertex.
Hence $P$ is a ``rational simple polytope" of Lerman and Tolman
\cite{LT}. In \cite{LT}, the authors showed that for any rational
simple polytope $P$, \footnote{In \cite{LT}, the authors consider
``labeled rational simple polytopes". In our case, all the labels
equal to 1.} there is a K\"ahler toric orbifold obtained by a
symplectic reduction construction, and this orbifold is isomorphic
to the toric variety associated to the dual fan of $P$ (See Theorem
1.7 of \cite{LT}). In our case, $P$ is the dual of $Q$, hence the
K\"ahler toric orbifold obtained by Lerman and Tolman is exactly
$X_Q$. In particular, $X_Q$ has a ``canonical" K\"ahler metric. By
the work of Guillemin and Abreu \cite{Gu}, \cite{Ab} (see also
\cite{CDG} and \cite{BGL}), this K\"ahler metric (called the
``Guillemin metric") has a nice expression using the combinatorial
data of $P$. When $X_Q$ is Fano, the K\"ahler form of the Guillemin
metric is in the Dolbeault class $2\pi c_1(X_Q)$ . Now we review
Guillemin and Abreu's result in our case.\\

In the following of this paper, we fix a lattice polytope $Q$ as
above and write $X:=X_Q$. We denote the complex torus by
$T_{\CC}\cong (\CC^\ast) $, which is an open dense subset of $X$,
with the standard coordinates $(z^1,\dots,z^n)$. We also denote by
$T$ the maximal torus subgroup
$T:=\{(e^{i\theta_1},\dots,e^{i\theta_n})|\ \theta_i\in \RR\}$. A
$T$-invariant function $\phi$ on $T_{\CC}$ can be viewed as a
function of $x=(x^1,\dots, x^n)$, where $x^i:=\log |z^i|^2$, so we
can identify it as a function in $\RR^n$. In this case, we have
$$\pbp \phi=\sqrt{-1}\sum_{i,j}\phi_{ij}\frac{dz^i}{z^i}\wedge\frac{d\bar z^j}{\bar z^j}.$$
In particular, when $\phi$ is a potential function of a K\"ahler
metric, we have $g_{i\bar j}=\phi_{ij}\frac{1}{z^i}\frac{1}{\bar
z^j}$, thus
$$\pbp \log \det(g_{i\bar j})=\pbp \log \det(\phi_{ij}).$$
\mbox{}

 Let the vertices of $Q$ be $\mathbf{n}^{(i)}\in N$,
$i=1,\dots, d$. Then we have
$$P=\{y\in M_{\RR}|~ \langle y, \mathbf{n}^{(i)}\rangle\geq -1,\ i=1,\dots, d\}.$$
In the interior of $P$ (denoted by $P^o$) we define \beq
\label{eqn:l_i} l_i(y):= \langle y, \mathbf{n}^{(i)}\rangle+1,\eeq
and \beq \label{eqn:u^0} u^0:=\sum_i l_i \log l_i. \eeq It is easy
to check that $u^0$ is strictly convex in $P^o$, and the gradient
map $Du^0$ is a diffeomorphism to $\RR^n$. We denote the Legendre
transform of $u^0$ by $\phi^0$, i.e.
$$\phi^0(x)=\langle Du^0(y), y \rangle-u^0(y)=\sum_i\big{(}l_i(y)-\log l_i(y)\big{)}-d,$$
where $x\in \RR^n$ and $y\in P^o$ are related by $x=Du^0(y)$.
 Then $\phi^0$ is a strictly convex smooth function on
$\RR^n$, and the Guillemin metric is given by $\omega_0=\pbp \phi^0
$.\\

We have the following properties of the Guillemin metric, which is
used in the next section.

\begin{lem} \label{lemma:logdetD2}
  We have $$|\log \det D^2 \phi^0 + \phi^0|\leq C $$  in
  $\RR^n$.
\end{lem}

\begin{proof}
 Note that the gradient map $D\phi^0$ is a diffeomorphism from
  $\RR^n$ to the interior of $P$, we can work on the polytope $P$. By the property of Legendre transforms, we know
  that for any $x\in \RR^n$, there is a unique $y\in P$ such that
  $y=D\phi^0(x)$ and $x=Du^0(y)$, moreover, we have
  $$\phi^0(x)=\langle Du^0(y), y \rangle-u^0(y),$$ and
  $$\det (D^2\phi^0)(x)=\det (D^2u^0)^{-1}(y).$$
  By (\ref{eqn:u^0}), it suffices to bound $\log \det (D^2 u^0)+\sum_i \log
  l_i$.

  By (\ref{eqn:l_i}) and (\ref{eqn:u^0}), we have
  $$(u^0)_{pq}=\sum_i \frac{\mathbf{n}_p^{(i)}\mathbf{n}_q^{(i)}}{l_i}.$$
  A direct computation shows that
  $$\det(D^2u^0)=\sum_{1\leq i_1<\dots < i_n \leq d}\frac{\det(\mathbf{n}^{(i_1)},\dots,\mathbf{n}^{(i_n)} )^2 }{l_{i_1}  \dots  l_{i_n}}.$$
  Then the lemma follows easily from this expression.\end{proof}
  \mbox{}

\begin{lem}\label{lem:asym}
  Let the vertices of $P$ be $p^{(1)},\dots,p^{(m)}$, and define
  $$v(x):=\max_k\{\langle x, p^{(k)}\rangle\}.$$
  Then we have
  $$|\phi^0-v|\leq C$$
  in $\RR^n$.
\end{lem}

\begin{proof}
  We also work on the polytope $P$. Let $y\in P$ be the unique
  point such that
  $y=D\phi^0(x)$ and $x=Du^0(y)$. So we have
  $$\phi^0(x)=\langle Du^0(y), y\rangle-u^0(y)\leq v(x)-u^0(y)\leq v(x)+C.$$

  On the other hand, suppose $v(x)=\langle x,p^{(k_0)}\rangle$, then
  we have
  $$v(x)=\langle \sum_i(1+\log l_i(y))\mathbf{n}^{(i)},p^{(k_0)}\rangle,$$
  thus
  $$v(x)-\phi^0(x)\leq \sum_i \log l_i(y)(1+\langle\mathbf{n}^{(i)}, p^{(k_0)}\rangle)+C'\leq C,$$
  since $1+\langle\mathbf{n}^{(i)}, p^{(k_0)}\rangle$ is
  nonnegative and the $l_i$'s are bounded from above on $P$.
\end{proof}
\mbox{}

\section{K\"ahler-Ricci soliton equation on toric Fano orbifolds }\label{section:equation}

We start with the general soliton equation. Let $\omega_0$ be the
K\"ahler form of the Guillemin metric $g^0_{i\bar j}$, and
$$\omega_g=\omega_0+\pbp \varphi,$$
where $\varphi$ is a smooth function on the regular part of $X$ such
that the pull-back of $\omega_g$ on any local uniformizing chart
$\tilde U$ extends to a K\"ahler form on $\tilde U$, that is,
$\omega_g$ is an orbifold K\"ahler form. For an orbifold holomorphic
vector field $v$, we have
$$L_v \omega_g=L_v\omega_0+\pbp v(\varphi).$$
By Hodge decomposition theorem on orbifolds (\cite{Bai2}), there is
a unique complex valued function $\theta_v$ such that
$$i_v \omega_0 =\sqrt{-1}\bar\partial \theta_v,$$
with the normalization condition $\int_X
\exp(\theta_v)\omega_0^n=\int_X \omega_0^n$. From this, we have
$$L_v \omega_0 = \pbp \theta_v.$$

Let $h$ be the Ricci potential of $\omega_0$, namely
$$Ric(\omega_0)-\omega_0=\pbp h,$$
with $\int_X \exp(h)\omega_0^n=\int_X \omega_0^n$. Then the equation
(\ref{eqn:soliton}) becomes
$$\pbp\big{(}\log\frac{\det(g^0_{i\bar j}+\varphi_{i\bar j})}{ \det (g^0_{i\bar j})}
-h+\varphi+\theta_v+v(\varphi)\big{)}=0,$$ thus
\beq\label{eqn:soliton-potential} \det(g^0_{i\bar j}+\varphi_{i\bar
j})=\det (g^0_{i\bar j})e^{h-\theta_v -v(\varphi)-\varphi}. \eeq
\mbox{}

Now we use the special symmetry of the toric variety. Suppose
$\omega_g$ is also $T$-invariant, then the restriction of $\varphi$
to $T_{\CC}$ can be viewed as a function in $\RR^n$, so we have
$$g_{i\bar j}=\phi_{ij}\frac{1}{z^i}\frac{1}{\bar z^j},$$
where $\phi:=\phi^0+\varphi$. Let $v:=\sum_i c_i z^i \frac{\p}{\p
z^i}$, then one can also check that
$$v(\varphi)=\sum_i c_i \varphi_i$$
and
$$\theta_v=\sum_i c_i \phi^0_i-c_v$$
for some constant $c_v$. Moreover, by Lemma \ref{lemma:logdetD2}, we
know that there is a constant $\tilde c$ such that
$$\log\det (D^2\phi^0)+\phi^0+h=\tilde c,$$thus we have

\beq \label{eqn:soliton-reduction} \det(\phi_{ij})=e^{-c-\phi-\sum_i
c_i\phi_i}, \eeq where the constant $c$ depends only on the initial
metric $g^0$ and the holomorphic vector field $v=\sum_i c_i z^i
\frac{\p}{\p z^i}$.\\

Now we are in a position to determine the constants $c_i$'s.

{\prop \label{prop:c_i}

 The necessary condition to  have a solution of
(\ref{eqn:soliton-reduction})  is  that the $c_i$'s satisfy the
equations \beq\label{eqn:c_i} \int_P y^i e^{\sum_l c_l y^l}
dy=0,\quad i=1,\dots, n. \eeq }
\begin{proof}
Let $\phi$ be a solution of (\ref{eqn:soliton-reduction}).  Then by
(\ref{eqn:soliton-reduction}) we have

$$\int_P y^i e^{\sum_l c_l y^l} dy= \int_{\RR^n} \phi_i e^{\sum_l c_l \phi_l}\det(\phi_{ij}) dx=
e^{-c}\int_{\RR^n} \frac{\p}{\p x^i} (e^{-\phi})dx=0.$$

\end{proof}

Since $P$ contains the origin in its interior, it is easy to see
that the $c_i$'s exist and are uniquely determined by
(\ref{eqn:c_i}), see, for example \cite{WZ} or \cite{Do}. Actually,
one needs only to consider the convex function $F(s_1,\dots,
s_n):=\int_P e^{\sum_l s_l y^l} dy$. Suppose $B_r(0)\subset P$ is a
ball in $P$, and set $\S_r:=\{y\in B_r(0) \big{|}~ y^1\geq
\frac{1}{2}|y|\}$, then $$F(s)\geq
\int_{\S_r}e^{\frac{1}{2}|s||y|}dy,$$ that is,  $F$ is proper. Hence
there is a unique minimum point $(c_1,\dots, c_n)$ of $F$,  which
satisfies
(\ref{eqn:c_i}).\\

 When
$c_i=0$ for all $1\leq i\leq n$, i.e., the barycenter of $P$ is the
origin, then the vector field $v=0$ and the soliton equation
(\ref{eqn:soliton-potential}) becomes the K\"ahler-Einstein
equation.

\begin{prop}
  The barycenter of $P$ is the origin if and only if the Futaki
  invariant of $X$ vanishes.
\end{prop}

\begin{proof}
  First, we use a theorem of Cox in \cite{Co}, namely, for the toric orbifold $X$ the maximal
  torus of $Aut(X)$ is exactly $T$. Let $\mathfrak{g}(X)$ be the Lie
  algebra of $Aut(X)$, consisting of holomorphic vector fields on
  $X$, and let the Cartan decomposition of $\mathfrak{g}(X)$ be
  $$\mathfrak{g}(X)=\mathfrak{h}(X)+\sum_i \CC w_i,$$
  where $\mathfrak{h}(X)$ is the Lie algebra of $T_\CC$, generated
  by $v_i=z^i\frac{\p}{\p z^i}$, $i=1,\dots, n$, and the $w_i$'s are
  the common eingenvectors of the adjoint actions $ad_v$ for $v\in
  \mathfrak{h}(X)$. For any $w_i$, there must be a $v\in \mathfrak{h}(X)$
  such that $ad_v(w_i)=\l_{v,i}w_i$ with $\l_{v,i}\neq 0$, for
  otherwise $w_i$ commutes with the whole of $\mathfrak{h}(X)$, contradicts with
  the fact that $\mathfrak{h}(X)$ is a maximal abelian subalgebra of
  $\mathfrak{g}(X)$.

    The Futaki invariant on a normal Fano orbifold is discussed in \cite{DT}. Now note that the Futaki invariant $\mathcal{F}$ vanishes on
    $[\mathfrak{g}(X),\mathfrak{g}(X)]$ as in the smooth case,
    we have
    $$\mathcal {F}(w_i)=\l_{v,i}^{-1}\mathcal{F}([v,w_i])=0.$$
    But a direct computation shows that up to a constant factor, $\mathcal{F}(v_i)$ is exactly $\int_P y^i
    dy$. The proposition follows from this fact.
\end{proof}
 \mbox{}

\section{Existence of K\"ahler-Ricci solitons}\label{section:existenc}

As in \cite{WZ}, we use the continuity method to  consider a family
of  equations,  \beq \label{eqn:continuity} \det(g^0_{i\bar
j}+\varphi_{i\bar j})=\det (g^0_{i\bar j})e^{h-\theta_v
-v(\varphi)-t\varphi} \eeq
 with parameter  $t\in [0,1]$. Then $\phi$ satisfies
the equation

\beq \label{eqn:continuity-toric} \det(\phi_{i j})=e^{-c-w-\sum_i
c_i\phi_i} \eeq
 in $\RR^n$, where
\beq\label{eqn:w} w=w_t:=t\phi+(1-t)\phi^0. \eeq  As in \cite{WZ}
and \cite{TZ1}, it suffices to obtain a uniform estimate for
$\phi-\phi^0$ when $t\in [\varepsilon_0,1]$.\\

The estimate is almost identical to that of \cite{WZ}, for readers'
convenience, we include it briefly here.

{\lem\label{lemma:m_t} Let $m_t:=\inf_{x\in \RR^n} w_t(x)$, then we
have
$$|m_t|\leq C$$
for some constant $C$ independent of $t$. }

\begin{proof}
 The proof is the same to that of \cite{WZ}. First, note that the image of the gradient map $D\phi$ is also
the interior of the polytope $P$. By the equation
(\ref{eqn:continuity-toric}) and the properties of Legendre
transform, we have \beq \int_{\RR^n} e^{-w}=\int_{\RR^n}
\det(\phi_{ij})e^{c+\sum_i c_i\phi_i}dx=e^c\int_P e^{\sum_i c_i
y^i}dy=:\beta. \eeq Since $|Dw|\leq d_0:=\sup\{ |y|\ \big{|}\ y\in
P\}$, we have
$$vol(B_1(x^t))e^{-m_t-d_0}\leq \beta,$$
thus $m_t\geq C$, for some constant $C$ independent of $t$.

Next we derive the upper bound of $m_t$. Let $A_\l:=\{x\in \RR^n
\big{|} w(x) \leq m_t+\l\}$. Then as in \cite{WZ}, we have
$vol(A_1)\leq C e^{\frac{m_t}{2}}$. Then by convexity of $w$, we
know that for any $\l>1$ we have $vol(A_\l)\leq C\l^n
e^{\frac{m_t}{2}} $, thus we can show that
$$\beta \leq C'e^{-\frac{m_t}{2}},$$
hence $m_t\leq C$.

\end{proof}

{\lem\label{lemma:x^t} Let $x^t\in \RR^n$ be the unique point such
that $w_t(x^t)=m_t$, then we have
$$|x^t|\leq C$$
for some constant $C$ independent of $t$. }

\begin{proof}
 First note that  $vol(A_1)\leq C$ by the proof of Lemma \ref{lemma:m_t}, and since $|Dw|\leq d_0$, there is a ball
 centered at $x^t$ with fixed size contained in $A_1$. If $A_1$ contains a point $x$ with $|x-x^t|$ large, then by convexity
of $A_1$, the volume of $A_1$ will also be large. So we can choose a
$R>0$ independent of $t$ such that $A_1\subset B_R(x^t)$.

Also by convexity of $w$, we have
$$|Dw|>\frac{1}{R}\quad {\text in\  } \RR^n\setminus B_R(x^t).$$
Hence for any $\varepsilon>0$ samll, we can find a sufficiently
large $R_{\varepsilon}$ (independent of $t$) such that
$$\int_{\RR^n\setminus B_{R_\varepsilon}(x^t) }e^{-w} dx \leq \varepsilon. $$

On the other hand, for any $\varepsilon>0$ small, we can find a
large constant $C>0$ such that if $|x^t|>C$, we have
$$\xi\cdot D\phi^0 >\frac{a_0}{2}   \quad {\text in\ } B_{R_{\varepsilon}}(x^t), $$
where $\xi=x^t/|x^t|$, and $a_0:=\inf\{|y|\big{|} y\in \p P\}$.
Hence for $\varepsilon$ sufficiently small, one has
$$\int_{\RR^n}\xi\cdot D\phi^0 e^{-w} dx >0.$$
However, by (\ref{eqn:c_i}) and (\ref{eqn:continuity-toric}), we
have \beqs
0 &=& \int_P y^i \exp(\sum_l c_ly^l) dy\\
  &=& \int_{\RR^n} \phi_i \exp(\sum_l c_l\phi_l) \det D^2\phi dx\\
  &=& e^{c} \int_{\RR^n} \phi_i e^{-w} dx\\
  &=& -\frac{1-t}{t}  e^{c} \int_{\RR^n} \phi^0_i e^{-w} dx.
\eeqs Thus $$ \int_{\RR^n} \xi\cdot D\phi^0 e^{-w} dx=0,$$ which is
a contradiction.

\end{proof}

{\prop \label{prop: upper} Let $\varphi=\varphi_t$, where $t\in
[\varepsilon_0,1]$, be a solution of (\ref{eqn:continuity-toric}),
then
$$\sup_X \varphi \leq C$$
for some constant $C$ independent of $t$. }

\begin{proof}
By Lemma \ref{lemma:m_t} and Lemma \ref{lemma:x^t}, we know that
$|w(0)|\leq C$, so $|\phi_0|\leq C$ for $t\in [\varepsilon_0,1]$.
From Lemma \ref{lem:asym},
 we have a function $v$, whose gragh is the asymptotical cone of the graph of $\phi^0$. Since $D\phi^0(\RR^n)=D\phi(\RR^n)$, we have
$$\phi(x)-\phi(0)\leq v(x)-v(0).$$
So we have
$$\varphi=\phi-\phi^0\leq v-\phi^0+\phi(0)-v(0).$$
Again by Lemma \ref{lem:asym}, we have $\sup_X
\varphi=\sup_{\RR^n}\varphi\leq C$.
\end{proof}

Next we need a Harnack type theorem to control the infimum of
$\varphi$. Here we use an idea of Donaldson \cite{Do}, to prove it
via the ordinary Sobolev imbedding theorem on $P$.

{\prop \label{prop: lower} Let $\varphi$ be as in Proposition
\ref{prop: upper}, then we have
$$\inf_X \varphi \geq -C$$
for some constant $C$ independent of $t$. }

\begin{proof}
 Let the Legendre transform of $\phi$ be $u$. By definition, we have
$$u(y)=\sup_{\tilde x\in \RR^n} (\tilde x\cdot y-\phi(\tilde x)).$$Then one can check easily that
$$\sup_{\RR^n}(\phi^0-\phi)=\sup_P (u-u^0).$$
Actually, suppose for $y\in P^o$, $x\in \RR^n$ is the unique point
such that $u(y)=x\cdot y -\phi(x)$, then we have \beqs u(y)-u^0(y)
&=& x\cdot y -\phi(x)-\sup_{\tilde x\in \RR^n} (\tilde x\cdot
y-\phi^0(\tilde x))\\
&\leq & x\cdot y -\phi(x)-x\cdot y +\phi^0(x)=
\phi^0(x)-\phi(x)\\
&\leq & \sup_{\RR^n}(\phi^0-\phi).\eeqs Thus we get $\sup_P
(u-u^0)\leq \sup_{\RR^n}(\phi^0-\phi)$, and the same argument
implies that $\sup_{\RR^n}(\phi^0-\phi)\leq \sup_P (u-u^0)$.

Now it suffices to bound $u$ on $P$.

The idea is to bound $\parallel Du\parallel_{L^p(P)}$ for $p>n$,
then by the Sobolev embedding theorem on $P$, we get the estimate of
${\text osc}_P u$.

Note that
$$\int_{P} |Du|^p dy= \int_{\RR^n} |x|^p \det (\phi_{ij}) dx\leq C\int_{\RR^n} |x|^p e^{-w} dx. $$
Take $R$ as in the proof of Lemma \ref{lemma:x^t}, then out of
$B_R(x^t)$, we have
$$w(x)\geq m_t+1+\frac{1}{R}|x-x^t|,$$
thus $$w(x)\geq \epsilon |x-x^t|-C\quad {\text in\ }\RR^n$$ for some
constants $\epsilon$ and $C$ independent of $t$. Now it is obvious
that we have
$$\parallel Du\parallel_{L^p(P)}\leq C.$$
Now by (7.45) of \cite{GT}, we have
$$\|u-u_P\|_{W^{1,p}}\leq C,$$
where $u_P:=\frac{1}{vol(P)}\int_P u dy$ is the average of $u$ over
$P$. Then since the boundary of $P$ is Lipschitz, we have the
Sobolev imbedding
$$\sup_P |u-u_P|\leq C,$$
and hence
$$osc_P u \leq 2C.$$
So the proposition is true. \end{proof}

Proposition \ref{prop: upper} and \ref{prop: lower} complete the
proof of Theorem \ref{thm:main}.

\begin{ex}
  Let $Q\subset N_\RR\cong \RR^2$ be a lattice polytope, whose vertices are
  $(1,0), (0,1)$ and $(-2,-1)$. Then the corresponding toric variety $X_Q$ is a Fano orbifold with
  one singular point which is an ordinary double point. Actually, $X_Q$ coincides with
  ``A-1" in Nakagawa's table on page 240 of \cite{Na}. One can
  check easily that $X_Q$ is a global quotient of $\CC P^2$. Note that the Fubini-Study
  metric descents to $X_Q$, but it is singular along a divisor.
  Actually, since the barycenter of the dual polytope $P$ is not the origin,
   the Futaki invariant of $X_Q$ is not zero, so $X_Q$ does
  not admit a K\"ahler-Einstein metric. However, by Theorem \ref{thm:main},
  $X_Q$ admits a K\"ahler-Ricci soliton metric.
\end{ex}
\mbox{}

Now we give an example of toric Fano variety with an invariant
K\"ahler-Einstein metric, whose anticanonical divisor is not
Cartier.

\begin{ex}
  Let $Q\subset N_\RR\cong \RR^2$ be a lattice polytope, whose vertices are
  $(-2,-1), (-2,1), (2,-1)$ and $(2,1)$. Then $X_Q$ is a toric Fano
  orbifold. $-K_{X_Q}$ is not Cartier but $-2K_{X_Q}$ is. One can also check
  that $X_Q$ is a global quotient of the sueface ``B-2" in Nakagawa's table.
   The dual polytope of $Q$ is
  $$P=\{y\in M_\RR |~ l_i(y)\geq 0,~i=1,2,3,4\},$$
  where $l_1(y)=-2y^1-y^2+1$, $l_2(y)=-2y^1+y^2+1$, $l_3(y)=2y^1-y^2+1$
  and $l_4(y)=2y^1+y^2+1$. Obviously, the barycenter of $P$ is the
  origin. By Theorem \ref{thm:main}, $X_Q$ admits a K\"ahler-Einstein
  metric.
\end{ex}
\mbox{}\\

\end{document}